\documentclass[12pt]{article}
\usepackage{amsfonts}
\usepackage{graphics}
\usepackage{amsmath}
\usepackage{amsthm}

\newtheorem{theorem}{Theorem}[section]

\newtheorem{lemma}{Lemma}[section]

\newtheorem{proposition}{Proposition}[section]

\theoremstyle{remark}

\begin{document}

\title{A note on the relation $\mathcal{J}$ in $le$-semigroups}
\author{Aida Shasivari\\
Polytechnic University of Tirana\\
Faculty of Mathematical Engineering\\
Tirana, Albania\\
aida.shasivari@yahoo.com\\
Elton Pasku\\
University of Tirana\\
Faculty of Natural Sciences\\
Tirana, Albania\\
elton.pasku@fshn.edu.al}
\date{}
\maketitle

\begin{abstract}
We prove that if $S$ is a $le$-semigroup in which left ideal elements commute (condition which is called $\mathbf{\Lambda}$), then any $\mathcal{J}$-class satisfying the Green condition is a subsemigroup of $S$. As a corollary of this we show that semisimple $le$-semigroups satisfying $\mathbf{\Lambda}$ are precisely those that decompose as a semilattice of left simple $\vee e$-semigroups which are in addition semisimple, intra-regular and satisfy $\mathbf{\Lambda}$. \newline
\end{abstract}

\section{Introduction and preliminaries}

The study of Green relations in ordered semigroups has been a reoccurring theme and a great corpus of results has been obtained through the last 30 years or so.  Relations $\mathcal{B}$, $\mathcal{Q}$ and $\mathcal{H}$ in $le$-semigroups are ''closer to the points" and seem to be more adequate to study the way regularity or intra-regularity is transmitted from points to subsets, but fail a key property their analogue $\mathfrak{H}$ in plain semigroups has, which says that any $\mathfrak{H}$-class containing the product of two arbitrary elements that belong to the class (commonly known as the Green condition) is a subgroup of the semigroup. Relations $\mathcal{L}$ and $\mathcal{R}$ also do not seem to have this property either but have many other interesting properties that resemble to their analogues in plain semigroups and are proved to be indispensable in the study of ordered semigroups. Relation $\mathcal{J}$ in ordered semigroup is proved fruitful in the study of intra-regularity in particular. For example Kehayopulu in \cite{intra-93} and Kehayopulu at alt. in \cite{intraP-93} have shown that intra-regular ordered semigroups are precisely those that decompose as a semilattice of simple subsemigroups. In the present paper we look for conditions under which the $\mathcal{J}$-classes of a $le$-semigroup satisfying the Green condition are subsemigroups. The $le$-semigroups we work with satisfy the condition that every two left ideal elements commute with each other and call this condition $\mathbf{\Lambda}$. This condition is not unnatural. In \cite{gms} Kehayopulu has shown that in $le$-semigroups the condition $\mathbf{\Lambda}$ together with left regularity are equivalent to $l \vartriangleright l$ which in turn means the left regularity and that the left ideal elements are at the same time right ideal elements. We also show that semisimple $le$-semigroups that satisfy $\mathbf{\Lambda}$ have a nice decomposition as a semilattice of left simple $\vee e$-semigroups which are semisimple, intra-regular and satisfy $\mathbf{\Lambda}$ as well. This result is in the same spirit as those in \cite{intra-93} and \cite{intraP-93}.

We give below some basic notions that will be used throughout the paper. A \textit{$po$-semigroup} (\textit{ordered semigroup}) is a triple $(S,\leq, \cdot)$ where $\leq$ is an order relation in $S$ and $\cdot$ is an associative multiplication in $S$ that satisfy the properties
\begin{equation*}
a \leq b \Rightarrow xa \leq xb \text{ and } ax \leq bx \text{ for all } x \in S.
\end{equation*}
If $(S,\leq, \cdot)$ is a $po$-semigroup possessing a greatest element $e$, then it will be called a $poe$-semigroup and it is called a $\vee e$-semigroup if $(S,\leq)$ is an upper semilattice. If in addition it happens that $(S,\leq)$ is a lattice (the meet $\wedge$ and the joint $\vee$ of any two elements of $S$ exists), then $S$ will be called a $le$-semigroup. The standard notation for the $le$-semigroup in this case is $\langle S, \cdot, \vee, \wedge \rangle$. Here the order relation is not made explicit but it is understood that $a \leq b$ iff $a \wedge b=a$. An element $x$ of a $poe$-semigroup $S$ is regular if $x \leq xex$, and it is intra-regular if $e \leq e x^{2}e$. $S$ is called regular (resp. intra regular) if every element of $S$ is regular (resp. intra-regular). A $poe$-semigroup $S$ is called semisimple if for all $x \in S$, $x \leq e xex e$. An element $t \in S$ is called semiprime if for every $a \in S$ such that $a^{2} \leq t$, then $a \leq t$. An element of $S$ is called a right (resp. left) ideal element if $xe \leq x$ (resp. $ex \leq x$). It is called an ideal element if it is both a left and a right ideal element of $S$. A $poe$-semigroup $S$ is called a left simple semigroup if $e$ is the only left ideal element of $S$. We say that a $poe$-semigroup $S$ decomposes as a semilattice of left simple $poe$-semigroups if there exists a semilattice $Y$ and a family $\{S_{\alpha}: \alpha \in Y\}$ of left simple $poe$-subsemigroups of $S$ such that:
\begin{itemize}
\item [1)] $S_{\alpha} \cap S_{\beta} = \emptyset$ for all $\alpha \neq \beta$ in $Y$,
\item[2)] $S=\cup \{S_{\alpha}: \alpha \in Y\}$,
\item[3)] $S_{\alpha} \cdot S_{\beta} \subseteq S_{\alpha \beta}$.
\end{itemize}

Let $\mathbf{S}$ stands for an $le$-semigroup $\langle S, \cdot, \vee, \wedge \rangle$. We define the map
\begin{equation*}
t: S \rightarrow S \text{ by setting } t(x)=exe \vee xe \vee ex \vee x,
\end{equation*}
and show that $t(x)$ is the least ideal element which is bigger than $x$. That $t(x)$ is an ideal element follows from the fact that
\begin{equation*}
(exe \vee xe \vee ex \vee x)\cdot e=exe^{e} \vee xe^{e} \vee exe \vee xe \leq exe \vee xe \vee ex \vee x,
\end{equation*}
and that
\begin{equation*}
e \cdot (exe \vee xe \vee ex \vee x)=e^{e}xe \vee exe \vee e^{e}x \vee ex \leq exe \vee xe \vee ex \vee x.
\end{equation*}
Also the definition of $t(x)$ shows that $t(x) \geq x$. It remains to show that if $\tau$ is an ideal element of $\mathbf{S}$ such that $x \leq \tau$, then $t(x) \leq \tau$. Indeed, 
\begin{align*}
t(x)&=exe \vee xe \vee ex \vee x \\
&\leq e \tau e \vee \tau e \vee e \tau \vee \tau && \text{ since $x \leq \tau$ } \\
&\leq \tau && \text{ since $\tau$ is an ideal element.}
\end{align*}
Now we are ready to define the relation $\mathcal{J}$ in $\mathbf{S}$. We say that two elements $x,y \in S$ are $\mathcal{J}$-related, written as $(x,y) \in \mathcal{J}$ if and only if $t(x)=t(y)$. In other words
\begin{equation*}
\mathcal{J}=\{(x,y) \in S \times S | t(x)=t(y) \}.
\end{equation*}
It is straightforward that $\mathcal{J}$ is an equivalence relation, hence we can talk about its $\mathcal{J}$-classes. A key property of the $\mathcal{J}$-classes is depicted in the following.
\begin{lemma} \label{kp}
Every $\mathcal{J}$-class $J$ of $\mathbf{S}$ has a unique ideal element $t(a)$ where $a$ is an arbitrary element of $J$.
\end{lemma}
\begin{proof}
First, for any $a \in J$ we show that $t(a) \in J$. Indeed, since $t(a)$ is an ideal element, $t(t(a))=t(a)$ and then the definition of $\mathcal{J}$ shows that $(a, t(a)) \in \mathcal{J}$ or equivalently, $t(a) \in J_{a}$. Finally, if $\tau \in J$ is an ideal element, then on the one hand $t(\tau)=\tau$ and on the other $t(\tau)=t(a)$ since $a \in J$, and so we get that $\tau=t(a)$.
\end{proof}
The element $t(a)$ of the previous lemma is called the \textit{representative ideal element} of $J$.

\section{The main results}

A $\mathcal{J}$-class is said to satisfy the \textit{Green condition} if there are $b,c \in J$ such that $bc \in J$. The representative ideal element plays a key role in regard with the Green condition as the following shows. 
\begin{theorem} \label{t-idemp}
A $\mathcal{J}$-class $J$ satisfies the Green condition if and only if its representative ideal element $\tau$ is an idempotent.
\end{theorem}
\begin{proof}
The if part is obvious so it remains to prove that if $b,c, bc \in J$, then $\tau=\tau^{2}$.  Since from lemma \ref{kp} $b,c \leq \tau$, then $bc \leq \tau^{2}$. Passing to ideal element generated from the both sides in the last inequality we get that $t(bc) \leq t(\tau^{2})$. But $\tau^{2}$ is an ideal element as well and so $t(bc) \leq \tau^{2}$. Recalling that $bc \in J_{\tau}$, we have $\tau \leq \tau^{2}$ which together with the obvious inequality $\tau^{2} \leq \tau$ imply the desired $\tau = \tau^{2}$.
\end{proof}

It is interesting to ask under what conditions a $\mathcal{J}$-class satisfying the Green condition forms a subsemigroup of $\mathbf{S}$. The answer to the case when the class is required to be a group is rather easy and is proved in proposition \ref{gc}. Before we need the following.
\begin{lemma} \label{lp}
Assume the $\mathcal{J}$-class $J$ is a subsemigroup of $\mathbf{S}$. If $\tau$ is the representative ideal element of $J$, then for every element $x \in J$
\begin{equation*}
\tau x \tau =\tau=exe \text{ and } \tau x e =\tau = e x \tau.
\end{equation*}
In particular we have $\tau e=\tau=e \tau$.
\end{lemma}
\begin{proof}
It is obvious that $\tau x \tau$ is a ideal element of $\mathbf{S}$. Since $J$ is assumed to form a subsemigroup and $x \in J$, then $\tau x \tau \in J$ and so by lemma \ref{kp} we get $\tau x \tau =\tau$. Further,
\begin{equation*}
\tau=\tau x \tau \leq \tau x e \leq exe \leq \tau,
\end{equation*}
proving that $\tau x e =\tau=exe$. In a similar fashion one can prove that $ex \tau=\tau=exe$. If we now replace $x$ by $\tau$ in the last two equalities and recall from theorem \ref{t-idemp} that $\tau=\tau^{2}$, then we obtain $\tau e=\tau = e \tau$.
\end{proof}

The following lemma shows that the above equalities are always ensured if we assume for $S$ to be regular and intra regular. 
\begin{proposition} \label{i+r}
Assume that $S$ is both regular and intra regular then for every $x \in S$ the ideal element $\tau$ generated from $x$ satisfies the equalities of lemma \ref{lp}.
\end{proposition}
\begin{proof}
Indeed,
\begin{align*}
\tau x \tau &= (exe \vee ex \vee xe \vee x)\cdot x \cdot (exe \vee ex \vee xe \vee x) \\
& \geq exe \cdot x \cdot exe \vee exe \cdot x \cdot ex \vee exe \cdot x \cdot xe \vee xe \cdot x \cdot exe\\
& \geq exe \vee ex \vee x \vee xe=\tau.
\end{align*}
On the other hand $\tau x \tau \leq \tau e \leq \tau$, and therefore we have $\tau x \tau =\tau$. Also we have
\begin{equation*}
\tau =\tau x \tau \leq \tau x e \leq exe \leq \tau,
\end{equation*}
proving that $\tau x e =exe =\tau$. The equality $e x \tau=\tau$ is proved by symmetry and the last two equalities of lemma \ref{lp} follow in the same way.
\end{proof}

The following is an analogue of proposition 2.3 of \cite{pp02}.

\begin{proposition} \label{gc}
A $\mathcal{J}$-class $J$ is a subgroup of $\mathbf{S}$ if and only if $J$ consists of a single idempotent element.
\end{proposition}
\begin{proof}
The if part is obvious. Now assume that $J$ is a subgroup of $\mathbf{S}$ and let $\tau$ be the representative ideal element of $J$, $\bar{\tau}$ its inverse and $i$ the unit element. We see first that $i$ is a ideal element. Indeed,
\begin{align*}
ei&=e \tau \bar{\tau} \\
&=\tau \bar{\tau} && \text{from lemma \ref{lp}}\\
&=i=ie && \text{in a similar fashion to above,}
\end{align*}
proving the claim. But $\tau$ is the only ideal element of $J$ so $\tau=i$. Further if $x \in J$, then
\begin{align*}
i&=\tau && \text{from above}\\
&=\tau x \tau && \text{from lemma \ref{lp}}\\
&=ixi=x && \text{since $i$ is the unit,}
\end{align*}
thus proving that $J=\{i\}$ is a singleton.
\end{proof}

In what follows we will consider $le$-semigroups that satisfy the condition $\mathbf{\Lambda}$: \textit{any two left ideal elements commute with each other}. We prove in our theorem \ref{main} that in such semigroups any $\mathcal{J}$-class that satisfies the Green condition is a subsemigroup. 
Assume we are given a $le$-semigroup $\langle S, \cdot, \vee, \wedge \rangle$ with the greatest element $\varepsilon$ and let $I$ be the set of ideal elements of $S$. For every $e \in I$ we let $]e]=\{ x \in S : x \leq e \}$. It is clear that $]e]$ is a subsemigroup of $S$ since if $x,y \in ]e]$, then $xy \leq ee \leq e \varepsilon \leq e$. With the order relation of $S$ restricted in $]e]$, the latter forms a $le$-semigroup with the greatest element $e$. We denote by $\mathcal{J}^{(e)}$ the relation $\mathcal{J}$ defined in $]e]$. With the above notations we have the following.
\begin{lemma} \label{idp}
If $e \in I$ is an idempotent, then $J_{e}^{(e)}$ is a subsemigroup. 
\end{lemma}
\begin{proof}
First, for every $x \in S$, $ex$ is left ideal element of $S$ since $\varepsilon(ex)=(\varepsilon e)x \leq ex$. Second, it is easy to see that for every $x \in J_{e}^{(e)}$ we have $e=exe$. Further, if $x, y \in J_{e}^{(e)}$, then we can write
\begin{align*}
e=(exe)(eye)&=e (xe)(ye)=(ex)e(ye) && \text{ $e$ is an idempotent}\\
& =e(ex)(ye) && \text{ from condition $\mathbf{\Lambda}$}\\
&= e(xy)e \leq e^{4}=e,
\end{align*}
from which follows easily that $xy \in J_{e}^{(e)}$. 
\end{proof}

\begin{theorem} \label{main}
If $S$ satisfies condition $\mathbf{\Lambda}$, then any $\mathcal{J}$-class satisfying the Green condition is a subsemigroup of $S$.
\end{theorem}
\begin{proof}
Assume that $J$ is a $\mathcal{J}$-class in $S$ satisfying the Green condition. Write $\tau$ for the representative ideal element of $\mathcal{J}$ which from lemma \ref{kp} has to be idempotent and from lemma \ref{lp} satisfies $\tau=\varepsilon \tau \varepsilon$. We claim that $J=J_{\tau} \subseteq J^{(\tau)}_{\tau}$. To see this we note first that for every $x \in J=J_{\tau}$, $\varepsilon \tau \varepsilon=\varepsilon x\varepsilon$. It follows that for the ideal element of $]\tau]$ generated from $x$ we have that, 
\begin{align*}
\tau x \tau \vee \tau x \vee x \tau \vee x &= (\varepsilon \tau \varepsilon)x (\varepsilon \tau \varepsilon) \vee (\varepsilon \tau \varepsilon)x \vee x (\varepsilon \tau \varepsilon) \vee x\\
&=\varepsilon \tau (\varepsilon x\varepsilon) \tau \varepsilon \vee (\varepsilon \tau \varepsilon)x \vee x (\varepsilon \tau \varepsilon) \vee x\\
&=\varepsilon \tau^{3} \varepsilon \vee (\varepsilon \tau \varepsilon)x \vee x (\varepsilon \tau \varepsilon) \vee x\\
&= \varepsilon \tau \varepsilon = \tau
\end{align*}
since $(\varepsilon \tau \varepsilon)x \vee x (\varepsilon \tau \varepsilon) \vee x \leq \tau=\varepsilon \tau \varepsilon$, which shows that $x \in J^{(\tau)}_{\tau}$ thus proving the inclusion $J=J_{\tau} \subseteq J^{(\tau)}_{\tau}$. Let $a, b \in J_{\tau}$. Since $J_{\tau} \subseteq J^{(\tau)}_{\tau}$ and $J^{(\tau)}_{\tau}$ is a subsemigroup from lemma \ref{idp} we have that $\tau =\tau ab \tau \vee \tau ab \vee ab \tau \vee ab$. But
\begin{align*}
\tau ab \tau \vee \tau ab \vee ab \tau \vee ab & \leq \varepsilon ab \varepsilon \vee \varepsilon ab \vee ab \varepsilon \vee ab\\
& \leq \varepsilon \tau^{2} \varepsilon \vee \varepsilon \tau^{2} \vee \tau^{2} \varepsilon \vee \tau^{2} && \text{ since $a, b \leq \tau$}\\
& = \varepsilon \tau \varepsilon \vee \varepsilon \tau \vee \tau \varepsilon \vee \tau  && \text{ since $\tau = \tau^{2}$}\\
&= \tau,
\end{align*}
which proves that $\tau=\varepsilon ab \varepsilon \vee \varepsilon ab \vee ab \varepsilon \vee ab$ or equivalently that $ab \in J_{\tau}$ concluding the proof.
\end{proof}

\begin{proposition} \label{sd}
Let $S$ be a $le$-semigroup with the greatest element $\varepsilon$. Then $S$ is semisimple and satisfies the condition $\mathbf{\Lambda}$, if and only if the decomposition of $S$ into $\mathcal{J}$-classes forms a semilattice of left simple $\vee e$-semigroups. Furthermore, the $\mathcal{J}$-classes of this decomposition are semisimple, intra-regular and satisfy the condition $\mathbf{\Lambda}$.
\end{proposition}
\begin{proof}
$\Rightarrow$. Since $S$ is semisimple, then Proposition 9 of \cite{tbt} implies that the set of ideal elements $I$ forms a semilattice and theorem \ref{main} implies that for any $\alpha \in I$ its $\mathcal{J}$-class $J_{\alpha}$ is a $\vee e$-semigroup with the greatest element $\alpha$. Next we prove that any ideal element $t$ in the $\vee e$-semigroup $J_{\alpha}$ is semiprime. This will be immediate if we show that $J_{\alpha}$ is left simple since this means in particular that $\alpha$ is the only ideal element in $J_{\alpha}$ which is obviously semiprime. Proposition 13 of \cite{tbt} would then imply that $J_{\alpha}$ is an intra-regular $\vee e$-semigroup. Now we prove that for every $\alpha \in I$, the only left ideal element in the $\vee e$-semigroup $J_{\alpha}$ is $\alpha$ itself. Indeed, if $b \in J_{\alpha}$ is a left ideal element there, then
\begin{align*}
b \leq \alpha &= (\alpha b) \alpha && \text{ from lemma \ref{lp}}\\
&= \alpha^{2} b=\alpha b && \text{ since $\mathbf{\Lambda}$ holds true}\\
& \leq b && \text{ $b$ is a left ideal element in $J_{\alpha}$,}
\end{align*}
proving that $b=\alpha$. This also shows that $J_{\alpha}$ is a semisimple $\vee e$-semigroup. Next we prove that every two left ideals of $J_{\alpha}$ commute with each other. This is immediate under our assumption for $\mathbf{\Lambda}$ if we show that left ideals of $J_{\alpha}$ are also left ideals of $S$. To prove the latter we let $b \in J_{\alpha}$ be a left ideal there. First we note that $\varepsilon b \leq \varepsilon \alpha \leq \alpha$ and that $\varepsilon b \in J_{\alpha}$ since
\begin{align*}
\varepsilon (\varepsilon b) \varepsilon \vee \varepsilon(\varepsilon b) \vee (\varepsilon b) \varepsilon \vee \varepsilon b &=\varepsilon b \varepsilon \\
&=\alpha && \text{ from lemma \ref{lp}}.
\end{align*}
Further
\begin{align*}
b \geq \alpha b&=(\alpha \varepsilon)b && \text{ since $b$ is a left ideal and from lemma \ref{lp}}\\
& = \alpha^{2} (\varepsilon b) && \text{ since $\alpha$ is an idempotent}\\
& = \alpha (\varepsilon b) \alpha && \text{ from condition $\mathbf{\Lambda}$}\\
& = \alpha (\varepsilon b) \alpha^{2} \geq \alpha (\varepsilon b)(\varepsilon b) \alpha && \text{ from $\varepsilon b \leq \alpha$}\\
& \geq \varepsilon b && \text{ since $J_{\alpha}$ is intra-regular,}
\end{align*}
proving that $b$ is a left ideal element of $S$. Finally we prove that if $\alpha, \beta$ are ideal elements of $S$, then $J_{\alpha} \cdot J_{\beta} \subseteq J_{\alpha \beta}$. Let $x \in J_{\alpha}$ and $y \in J_{\beta}$, hence from lemma \ref{lp} $\varepsilon x \varepsilon =\alpha$ and $\varepsilon y \varepsilon =\beta$ and want to prove that $\varepsilon xy \varepsilon =\alpha \beta$. Indeed,
\begin{align*}
\alpha \beta \geq \varepsilon (\alpha \beta) \varepsilon \geq \varepsilon (x y) \varepsilon &=\varepsilon(\varepsilon x)(y \varepsilon) && \text{ since $\varepsilon$ is an idempotent}\\
&= (\varepsilon x) \varepsilon( y\varepsilon ) && \text{ from condition $\mathbf{\Lambda}$}\\
&= \varepsilon(x \varepsilon y) \varepsilon,
\end{align*}
hence we have
$$\alpha \beta \geq \varepsilon (x y) \varepsilon \geq \varepsilon(x \varepsilon y) \varepsilon.$$
Further we see that
$$ \varepsilon (\varepsilon(x \varepsilon y) \varepsilon) \varepsilon=( \varepsilon x \varepsilon) (\varepsilon y \varepsilon)=\alpha \beta,$$
therefore $\varepsilon (x y) \varepsilon =\alpha \beta$ showing that $xy \in J_{\alpha \beta}$. \newline
$\Leftarrow$. Since each $\mathcal{J}$-class is a subsemigroup, then its representative ideal element is an idempotent and $S$ is semisimple. Before we show that two arbitrary left ideal elements of $S$ commute we observe that if $b \in J_{\beta}$ is a left ideal element of $S$ and $\beta$ is the representative ideal element, then $b$ is a left ideal element in the $\vee e$-semigroup $J_{\beta}$ since $\beta b \leq \varepsilon b \leq b$.
Let now $b$ and $c$ be left ideal elements of $S$ and let $J_{\beta}$ and $J_{\gamma}$ be the $\mathcal{J}$-classes of $S$ containing $b$ and $c$ respectively where $\beta$ and $\gamma$ are the respective ideal elements. Since the $\mathcal{J}$-classes form a semilattice and since $S$ is semisimple, then 
$$bc \in J_{\beta}J_{\gamma} \subseteq J_{\beta \gamma}=J_{\beta \wedge \gamma},$$
and similarly
$$cb \in J_{\gamma}J_{\beta} \subseteq J_{\gamma \beta}=J_{\gamma \wedge \beta}.$$
So there are left ideal elements of $S$, namely $bc$ and $cb$ belonging to the same $\mathcal{J}$-class $J_{\beta \wedge \gamma}$, hence $bc=cb$ since the class is left simple. This proves that $S$ satisfies condition $\mathbf{\Lambda}$.
\end{proof} 
As we mentioned in the introduction, the result of proposition \ref{sd} is close to that of theorem 1 of \cite{intra-93} which holds true for ordered semigroups in general and that identifies intra-regular ordered semigroups with semilattices of simple semigroups. Taking into account this similarity it is reasonable to raise the following question. \textit{Is there any connection between intra-regular $le$-semigroups and $le$-semigroups which are semisimple and satisfy the condition $\mathbf{\Lambda}$?} We note that on the one direction it is clear that intra-regular $le$-semigroups are semisimple (see \cite{intra-80}) and on the other direction it is not difficult to see that semisimple $le$-semigroups that satisfy $\mathbf{\Lambda}$ are intra-regular.

\end{document}